\documentclass[a4paper,reqno,12pt]{amsart}
\usepackage[T1]{fontenc}
\usepackage[utf8]{inputenc}
\usepackage[english]{babel}

\usepackage{amsmath}
\usepackage{amsthm}
\usepackage{amssymb}
\usepackage{mathtools}
\usepackage{amsmath}
\usepackage{amsfonts}
\usepackage{mathrsfs}
\usepackage{esint}
\usepackage{yfonts}
\usepackage{quoting}
\usepackage{graphics}
\usepackage{booktabs}
\usepackage{bm}
\usepackage{bbm}
\usepackage{lmodern}
\usepackage{stmaryrd}
\usepackage{caption}
\usepackage{enumitem}			% Per elenchi numerati
\usepackage{tikz}
\usepackage{bookmark}
\usepackage{pgfplots}
\usepgfplotslibrary{fillbetween}
\usepackage[mathscr]{euscript}
\usepackage{orcidlink}
\usepackage[cal=esstix]{mathalpha}

\usepackage{ulem}

\usepackage{hyperref}
\hypersetup{
	colorlinks=true,
	linkcolor=blue,
	citecolor=blue,
	urlcolor=black,
	linktoc=all
}

\usepackage[margin=1.1in]{geometry}

\quotingsetup{font=small}

\theoremstyle{plain}
\newtheorem{proposition}{Proposition}[section]

\theoremstyle{plain}
\newtheorem{theorem}{Theorem}[section]
\numberwithin{equation}{section}	% Numerazione equazioni relativa alla sezione

\theoremstyle{plain}
\newtheorem{lemma}[theorem]{Lemma}

\theoremstyle{plain}
\newtheorem{corollary}{Corollary}[theorem]

\theoremstyle{definition}
\newtheorem{remark}{Remark}[section]

\theoremstyle{definition}
\newtheorem{example}{Example}[section]

\DeclarePairedDelimiter{\abs}{\lvert}{\rvert}		%	Valore assoluto
\newcommand{\numberset}{\mathbb}
\newcommand{\N}{\numberset{N}}			%	Naturali
\newcommand{\R}{\numberset{R}}			%	Reali
\newcommand{\sfera}{\numberset{S}}		%	Sfera
\newcommand{\Vg}{V_{\! g}}

\def\Xint#1{\mathchoice
	{\XXint\displaystyle\textstyle{#1}}%
	{\XXint\textstyle\scriptstyle{#1}}%
	{\XXint\scriptstyle\scriptscriptstyle{#1}}%
	{\XXint\scriptscriptstyle\scriptscriptstyle{#1}}%
	\!\int}
\def\XXint#1#2#3{{\setbox0=\hbox{$#1{#2#3}{\int}$ }
		\vcenter{\hbox{$#2#3$ }}\kern-.6\wd0}}

\def\dashint{\Xint-}

\usepackage{enumitem}

\allowdisplaybreaks

\begin{document}
	
	\title[The Liouville equation on Riemannian surfaces]{The Liouville equation on Riemannian surfaces: \\ the role of volume growth \\ in classification and rigidity results}
	
	\author[Giulio Ciraolo]{Giulio Ciraolo \orcidlink{0000-0002-9308-0147}}
	\address[]{Giulio Ciraolo. Dipartimento di Matematica ‘Federigo Enriques’, Università degli Studi di Milano, Via Cesare Saldini 50, 20133, Milan, Italy}
	\email{giulio.ciraolo@unimi.it}
	
	\author[Alberto Farina]{Alberto Farina \orcidlink{0009-0003-6455-9161}}
	\address[]{Alberto Farina. LAMFA, CNRS UMR 7352, Université de Picardie Jules Verne, Rue Saint Leu 33, 80039, Amiens, France}
	\email{alberto.farina@u-picardie.fr}
	
	\author[Michele Gatti]{Michele Gatti \orcidlink{0009-0002-6686-9684}}
	\address[]{Michele Gatti. Dipartimento di Matematica ‘Federigo Enriques’, Università degli Studi di Milano, Via Cesare Saldini 50, 20133, Milan, Italy}
	\email{michele.gatti1@unimi.it}
	
	\date{\today}
	
	\begin{abstract}
		We study the Liouville equation~$-\Delta_g u = e^u$ on a complete, connected, non-compact, boundaryless  Riemannian surface~$(M, g)$ with non-negative Ricci curvature. Assuming only some asymptotic lower bound on the solution, we establish classification results for both the solutions and the ambient manifold, discussing also their optimality. Our results reveal a close connection between the volume growth of the manifold and the classification of both the solutions and the underlying manifold.
	\end{abstract}
	
	\maketitle
	
	% -------------------------------------------------------------------
	% INTRO
	
	\section{Introduction}
	\label{sec:intro}
	
	In a two-dimensional Riemannian manifold~$(M,g)$, we consider the Liouville equation
	\begin{equation}
		\label{eq:Liou}
		-\Delta_g u = e^u \quad\text{in } M.
	\end{equation}
	Here and throughout the manuscript, we adopt the convention~$\Delta_g u \coloneqq \mathrm{div}_{g} (\nabla u)$. This equation arises naturally in geometric analysis in the context of conformal changes of metrics on two-dimensional manifolds. More specifically, it describes the transformation of a metric with vanishing Gaussian curvature into one with positive constant Gaussian curvature. Moreover, the study of~\eqref{eq:Liou} is also motivated by several physical applications -- see, for instance,~\cite{berchio,CaiLaiReview} and references therein for more details.
	
	In the Euclidean case, Liouville~\cite{liou} proved that any solution to~\eqref{eq:Liou} can be represented in terms of a meromorphic function. More than a century later, Chen \& Li~\cite{cl} classified the entire solutions of~\eqref{eq:Liou} in~$\R^2$ under the finite-mass assumption
	\begin{equation}
		\label{eq:fmass}
		\int_{\R^2} e^u \, dx < +\infty.
	\end{equation}
	In particular, they proved that every entire solution enjoying~\eqref{eq:fmass} is radially symmetric with respect to some point and is given by one of the \textit{bubbles} 
	\begin{equation}
		\label{eq:bub}
		U[z,\lambda](x) \coloneqq \log\frac{8\lambda^2}{\left(1+\lambda^2\,\abs*{x-z}^2\right)^{2}}.
	\end{equation}
	Here,~$\lambda>0$ is the scaling parameter and~$z \in \R^2$ is the center of the bubble. 
	
	In~\cite{cl}, the finite-mass assumption~\eqref{eq:fmass} is used to derive the optimal asymptotic behavior of solutions at infinity, which in turn enables the application of the moving planes method to obtain the classification result. It is also worth mentioning that the classification result in~\cite{cl} was subsequently recovered by Chou \& Wan~\cite{CW1994} and Hang \& Wang~\cite{hw} using complex analysis techniques. Moreover, Chanillo \& Kiessling~\cite{CK1995} provided an alternative proof based on a combination of a Pohozaev-type identity and the isoperimetric inequality. See also the survey~\cite{CaiLaiReview} for an overview on this topic.
	
	More recently, still within the Euclidean framework, Eremenko, Gui, Li, \& Xu~\cite{emer} used methods from complex analysis to completely classify the entire solutions of~\eqref{eq:Liou} that are bounded from above.

	In the Riemannian setting, Catino \& Monticelli~\cite{cm} considered complete non-compact surfaces with non-negative Gaussian curvature, assuming both the finiteness of the mass, namely 
	\begin{equation}
	\label{eq:finite-mass-manifold}
		\int_{M} e^u \, d\Vg < + \infty, 
	\end{equation}	
	and the lower bound
	\begin{equation}
	\label{eq:asymp_catino}
		u(x) \geq -4\log r(x) - 2\gamma\log\log r(x)  \quad\text{as } r(x) \to +\infty,
	\end{equation}	
	for some~$\gamma \in [0,1)$, where~$r$ denotes the geodesic distance from a reference point in the manifold. Under these assumptions, they proved that~$(M,g)$ is isometric to the Euclidean plane endowed with its standard metric and, consequently, that the solution must be a bubble of the form~\eqref{eq:bub}.
	
	This result was later improved by Cai \& Lai~\cite{cai-lai}, who replaced~\eqref{eq:asymp_catino} with
	\begin{equation*}
		u(x) \geq -4\log r(x) + o(\log r(x)) \quad\text{as } r(x) \to +\infty,
	\end{equation*}	
	while still assuming the finite-mass condition~\eqref{eq:finite-mass-manifold}.
	
	In~\cite{cami}, Ciraolo, Farina \& Polvara established the first classification result without assuming the finite-mass condition. Exploiting a~$P$-\textit{function approach}, they studied the Liouville equation~\eqref{eq:Liou} on two-dimensional, complete non-compact Riemannian manifolds with non-negative Ricci curvature, assuming only the lower bound
	\begin{equation}
	\label{eq:asymp_CFP}
		u(x) \geq -4\log r(x) - 2\log F(r(x)) \quad\text{as } r(x) \to +\infty,
	\end{equation}
	where~$F$ is a positive non-decreasing function satisfying 
	\begin{equation*}
		\int_c^{+\infty} \frac{dt}{t F(t)} = + \infty \quad\text{for some } c \geq 1.
	\end{equation*}
	Finally, by exploiting the~$P$-function approach and the analysis in~\cite{cami}, Ou~\cite{ou-liou} obtained the classification result under the lower bound
	\begin{equation}
	\label{eq:asymp_Ou}
		u(x) \geq - 4 \log r(x) - 4\alpha \log F(r(x)) \quad\text{as } r(x) \to +\infty,
	\end{equation}
	for some~$\alpha \geq 0$.
	
	In the present manuscript, we further develop the approach introduced in~\cite{cami} to optimally improve upon the main results of~\cite{cami,ou-liou} concerning~\eqref{eq:Liou}, as we describe in the following subsection. Our approach will establish a clear connection between the volume growth of the manifold~$M$, the classification of solutions to~\eqref{eq:Liou}, and the rigidity on the ambient manifold~$M$.
	
	% ----------------------------------
	% MAIN RES
	
	\subsection{Main results}
	
	In order to present our results clearly, we first fix some notation.
	Let~$(M, g)$ be a two-dimensional complete smooth Riemannian manifold with~$\mathrm{Ric}\geq 0$, and fix an origin~$\mathcal{o} \in M$. We denote by~$B_R \coloneqq B_R(\mathcal{o})$ the geodesic ball centered at~$\mathcal{o}$ with radius~$R>0$ and by~$V(R) \coloneqq \Vg(B_R)$ its volume. Moreover, we denote by~$r(x) \coloneqq \mathrm{dist}(x,\mathcal{o})$ the geodesic distance from~$\mathcal{o}$. For~$R>0$, we also set
	\begin{equation}
	\label{eq:defq}
		q(R) \coloneqq \frac{R^2}{V(R)},
	\end{equation}
	which is non-decreasing by the Bishop-Gromov volume comparison theorem. In addition, we have
	\begin{equation}
	\label{eq:vg}
	\frac{V(2R)}{V(R)}	\leq 4 \quad\text{and}\quad V(R) \leq \pi R^2.
	\end{equation}
	
	With this notation in place, the main result of this manuscript reads as follows.
 	
 	\begin{theorem}
 	\label{th:main}
 		Let~$(M, g)$ be a complete, connected, non-compact, boundaryless Riemannian manifold of dimension~$n = 2$ with~$\mathrm{Ric} \geq 0$. Let~$u \in C^2(M)$ be a solution to~\eqref{eq:Liou} and assume that there exist~$\beta \in (0,1)$ and~$F \colon [0,+\infty) \to (0,+\infty)$ non-decreasing and such that
 		\begin{equation}
 		\label{eq:hF}
 			\int_{c}^{+\infty} \frac{dt}{tF(t)} = +\infty \quad\text{for some } c \geq 1.
 		\end{equation}
 		Moreover, suppose that
 		\begin{equation}
 		\label{eq:bbu}
 			u(x) \geq - 4 \log r(x) - q^\beta(r(x)) F^\beta(r(x)) \quad\text{for } r(x) \geq c.
 		\end{equation}
 		Then,~$(M,g)$ is isometric to~$(\R^2,g_{\R^2})$ and~$u$ is a bubble of the form~\eqref{eq:bub}.
 	\end{theorem}
 	
 	The novelty of Theorem~\ref{th:main} is twofold. First, in contrast to the asymptotic lower bounds~\eqref{eq:asymp_CFP}--\eqref{eq:asymp_Ou} used in~\cite{cami,ou-liou}, respectively, our result allows for a power-like behavior of the remainder term in~\eqref{eq:bbu}. Second, the lower bound~\eqref{eq:bbu} explicitly involves the volume ratio~$q$, defined in~\eqref{eq:defq}. Consequently, Theorem~\ref{th:main} establishes a clear connection between the volume growth of the manifold and the classification of both solutions to~\eqref{eq:Liou} and the ambient manifold itself.
 	
 	We also establish the optimality of Theorem~\ref{th:main} on every two-dimensional smooth model manifold. More specifically, we prove that~$\beta=1$ is not allowed in Theorem~\ref{th:main}.
 	
 	\begin{corollary}
 		\label{cor:c3}
		For every two-dimensional smooth model manifold $(M,g)$ satisfying~$\mathrm{Ric} \geq 0$ and~$\mathrm{Ric} \not \equiv 0$, where~$g = dr^2+ \psi^2(r)\, d\theta^2$ and~$\psi \in C^\infty([0,+\infty))$, there exist~$\gamma>0$, depending only on~$\psi$, and a smooth radial solution to~\eqref{eq:Liou} such that
 		\begin{equation*}
 			u(r) \geq - 4 \log r - \gamma q(r) \log r \quad\text{as } r \to +\infty.
 		\end{equation*}
 		As a consequence,~$\beta=1$ is not allowed in Theorem~\ref{th:main}.
 	\end{corollary}
 	
 	From Theorem~\ref{th:main}, we deduce some consequences presented in the following results.
 	
 	\begin{corollary}
 	\label{cor:c1}
 		Let~$(M, g)$ be a complete, connected, non-compact, boundaryless Riemannian manifold of dimension~$n = 2$ with~$\mathrm{Ric} \geq 0$. Let~$u \in C^2(M)$ be a solution to~\eqref{eq:Liou} and assume that there exist~$\beta \in (0,1)$ and~$F \colon [0,+\infty) \to (0,+\infty)$ non-decreasing satisfying~\eqref{eq:hF}. Moreover, suppose that
 		\begin{equation}
 			\label{eq:bbu-1}
 			u(x) \geq - 4 \log r(x) - F^\beta(r(x)) \quad\text{for } r(x) \geq c.
 		\end{equation}
 		Then,~$(M,g)$ is isometric to~$(\R^2,g_{\R^2})$ and~$u$ is a bubble of the form~\eqref{eq:bub}.
 	\end{corollary}
 	
 	We note that Corollary~\ref{cor:c1} significantly improves upon Theorem~1.2 in~\cite{cami} and Theorem~1.1 in~\cite{ou-liou}. In fact, the assumption in~\cite{ou-liou} is precisely~\eqref{eq:asymp_Ou}
 	for some~$\alpha \geq 0$ and~$F$ as in Corollary~\ref{cor:c1}. The condition~\eqref{eq:asymp_Ou} is more restrictive than~\eqref{eq:bbu-1}, since, for every~$\alpha \geq 0$ and~$\beta \in (0,1)$, the function~$t \mapsto 4\alpha \log F(t)$ has slower growth than~$t \mapsto F^\beta(t)$. Moreover, Corollary~\ref{cor:c1} is optimal in the sense that~$\beta = 1$ is not allowed. In this case, for~$F(t) = \varepsilon \log t$ with~$\varepsilon>0$, condition~\eqref{eq:bbu-1} would read as
 	\begin{equation*}
 		u(x) \geq - \left(4+\varepsilon\right) \log r(x) \quad\text{for } r(x) \geq c.
 	\end{equation*}
 	On the other hand, by Theorem~2 in~\cite{cai-lai}, we know that for every~$\gamma > 4$ there exists a non-flat~$(M,g)$, conformal to~$(\R^2,g_{\R^2})$, supporting a solution of~\eqref{eq:Liou} with finite mass, in the sense of~\eqref{eq:finite-mass-manifold}, such that~$u(x) \sim - \gamma \log r(x)$, as~$r(x) \to +\infty$.
 	
 	We point out that, after this work was completed, Cai~\cite{cai-nlap} showed that the conclusion of Corollary~\ref{cor:c1} holds true under the slightly more general condition
 	\begin{equation*}
 		u(x) \geq - 4 \log r(x) - o(\log r(x)) \quad\text{as } r(x) \to +\infty.
 	\end{equation*}
 	
 	\begin{corollary}
 		\label{cor:c2}
 		Let~$(M, g)$ be a complete, connected, non-compact, boundaryless Riemannian manifold of dimension~$n = 2$ with~$\mathrm{Ric} \geq 0$. Let~$u \in C^2(M)$ be a solution to~\eqref{eq:Liou} and assume that there exist~$\beta \in (0,1)$ and~$F \colon [0,+\infty) \to (0,+\infty)$ non-decreasing satisfying~\eqref{eq:hF}. Moreover, suppose that
 		\begin{equation}
 			\label{eq:bbu-2}
 			u(x) \geq  - r^\beta(x) F^\beta(r(x)) \quad\text{for } r(x) \geq c.
 		\end{equation}
		Then, it follows that
		\begin{equation}
		\label{eq:limsup}
			\limsup_{R\to +\infty} \frac{V(R)}{R^{1+\varepsilon} F^{\varepsilon}(R)} = + \infty \quad\text{for every } \varepsilon \in [0,1-\beta).
		\end{equation}
		As a consequence,~$(M,g)$ is conformal to~$(\R^2,g_{\R^2})$.
 	\end{corollary}
 	
 	We point out that Corollary~\ref{cor:c2} significantly improves upon Theorem~1.3 in~\cite{cami} and Theorem~1.2 in~\cite{ou-liou}. Moreover, also in this case, Corollary~\ref{cor:c2} is optimal in the sense that~$\beta=1$ in~\eqref{eq:bbu-2} cannot be attained, as follows from the examples in Section~\ref{sec:ex}. \newline
 	
 	Finally, we mention that the quasilinear case of the~$n$-Liouville equation with~$n \geq 3$ has also been studied in the literature, both in the Euclidean and Riemannian settings -- see, for instance,~\cite{cmr-nLiou,cirli-esp,esposito} and the references therein. In this paper, however, we focus exclusively on the classical Liouville equation~\eqref{eq:Liou}.
 	
 	% ----------------------------------
 	% STRUC
 	
 	\subsection{Structure of the paper}
 	
 	In Section~\ref{sec:prel}, we recall some preliminary facts that will be used in the proofs of the results in Section~\ref{sec:proof}. Section~\ref{sec:mmanif} contains a brief intermezzo on concave functions and smooth model manifolds, followed by an analysis of radial solutions on these manifolds in Section~\ref{sec:radsol}. Finally, in Section~\ref{sec:ex} we provide some examples showing the optimality of Corollary~\ref{cor:c2}.
 	
 	% -------------------------------------------------------------------
 	% KNOWN FAC
	
	\section{A few preliminary facts}
	\label{sec:prel}
	Following~\cite{cami}, we introduce the auxiliary function~$v \coloneqq e^{-\frac{u}{2}}$, which solves
	\begin{equation}
		\label{eq:eq-v}
		\Delta_g v = P \quad\text{in } M,
	\end{equation}
	where the~$P$-function is given by
	\begin{equation}
		\label{eq:defP}
		P \coloneqq \frac{1}{v} \left(\abs*{\nabla v}_g^2 + \frac{1}{2}\right) \!.
	\end{equation}
	
	The following result has been established in~\cite{cami,ou-liou}. 
	
	\begin{lemma}
		\label{lem:DP-Dln}
		Let~$(M,g)$ be a Riemannian manifold of dimension~$n=2$ with~$\mathrm{Ric} \geq 0$. Let~$v \in C^2(M)$ be a solution to~\eqref{eq:eq-v} and let~$P$ be defined by~\eqref{eq:defP}. Then, we have
		\begin{equation*}
			\Delta_g P \geq 0 \quad\text{and}\quad \Delta_g \log P \geq 0.
		\end{equation*}
	\end{lemma}
	
	 \begin{remark}
	 	We point out that, by standard elliptic regularity theory, any classical solution~$u \in C^2(M)$ to~\eqref{eq:Liou} is in fact of class~$C^{3,\alpha}_{\mathrm{loc}}(M)$. Since~$v = e^{-\frac{u}{2}}$, we consequently also have~$v \in C^3(M)$. In particular, this allows us to apply the results in~\cite{cami,ou-liou}, which were originally stated for functions~$u,v \in C^3(M)$.
	\end{remark}
	
	We now prove a Liouville-type result \textit{à la} Karp, which is crucial in the following.
	
	\begin{lemma}
	\label{lem:karp}
		Let~$(M, g)$ be a complete, non-compact Riemannian manifold of dimension~$n \geq 2$ and let~$F \colon [0,+\infty) \to (0,+\infty)$ be a non-decreasing function satisfying~\eqref{eq:hF}. Suppose that~$U \in C^2(M)$ satisfies~$\Delta_g U \geq 0$ in~$M$ and
		\begin{equation}
		\label{eq:hU}
			\int_{B_R} U_{+}^\theta \, d\Vg \leq R^2 F(R) \quad \text{as } R \to  +\infty 
		\end{equation}
		for some~$\theta>1$, where~$U_{+} \coloneqq \max\left\{U,0\right\}$. Then,~$U_{+}$ is constant in~$M$.
	\end{lemma}
	\begin{proof}
		For every~$\varepsilon>0$ and~$t \in \R$, let
		\begin{equation*}
		       \varphi_\varepsilon(t) \coloneqq \sqrt[4]{t_{+}^4+\varepsilon^4} -\varepsilon.
		\end{equation*}
		Then,~$\varphi_\varepsilon \in C^2(\R)$ is non-decreasing and convex. Define~$U_\varepsilon \coloneqq 
		\varphi_\varepsilon(U)$ and observe that
		\begin{equation}
		\label{eq:Uep}
			0 \leq U_\varepsilon \leq U_{+} \quad\text{and}\quad U_\varepsilon \to U_{+} \quad\text{in } M \text{ as } \varepsilon \to 0.
		\end{equation}
		Moreover, a direct computation reveals that
		\begin{equation*}
			\Delta_g U_\varepsilon = \varphi_\varepsilon'(U) \,\Delta_g U + \varphi_\varepsilon''(U) \,\abs*{\nabla U}_g^2 \geq 0 \quad\text{in } M
		\end{equation*}
		and, in light of~\eqref{eq:hU}--\eqref{eq:Uep}, we also have
		\begin{equation*}
			\int_{B_R} U_{\varepsilon}^\theta \, d\Vg \leq \int_{B_R} U_{+}^\theta \, d\Vg \leq R^2 F(R) \quad \text{as } R \to + \infty. 
		\end{equation*}
		As a consequence, Theorem~2.2 in~\cite{karp} implies that~$U_\varepsilon$ is constant in~$M$, that is~$U_\varepsilon = c_\varepsilon$ for some~$c_\varepsilon \geq 0$. Finally, letting~$\varepsilon \to 0$, from~\eqref{eq:Uep} we infer the desired conclusion.
	\end{proof}
	
	% -------------------------------------------------------------------
	% PROOF
	
	\section{Proof of the main results}
	\label{sec:proof}
	
	We start by proving the main result of this manuscript.
	
	\begin{proof}[Proof of Theorem~\ref{th:main}]
		We aim to prove that~$P$ is constant in~$M$, the conclusion then follows as in the proof of Theorem~1.2 in~\cite{cami}.
		
		We start by observing that, since~$v = e^{-\frac{u}{2}}$, assumption~\eqref{eq:bbu} implies that
		\begin{equation}
			\label{eq:buv}
			v(x) \leq r^2(x) e^{q^\beta(r(x)) F^\beta(r(x))} \quad\text{for } r(x) \geq c.
		\end{equation}
		
		We first aim to convert~\eqref{eq:buv} into an integral estimate for~$P$. To this end, for any~$R>0$, we consider a standard  cut-off function~$\eta$, namely, a function~$\eta \in C^{0,1}_c(M)$ satisfying
		\begin{equation}
			\label{eq:cutoff}
			0 \leq \eta \leq 1 \text{ in } M, \; \eta=1 \text{ in } B_R, \; \eta=0 \text{ in } M \setminus B_{2R} \text{ and } \abs{\nabla \eta}_g \leq \frac{1}{R} \text{ in } B_{2R} \setminus B_{R}.
		\end{equation}
		We test~\eqref{eq:eq-v} with~$\psi=\eta^2$ and integrate by parts, obtaining
		\begin{equation*}
			\int_{M}  v^{-1} \,\abs*{\nabla v}_g^2 \eta^2 \, d \Vg + \frac{1}{2} \int_{M} v^{-1} \eta^2 \, d \Vg = \int_{M} P \eta^2 \, d \Vg = - 2 \int_{M} \left\langle \nabla v, \nabla \eta \right\rangle_g \eta \, d\Vg,
		\end{equation*}
		where we exploited~\eqref{eq:defP}. Applying Young's inequality yields
		\begin{equation*}
			\int_{M} v^{-1} \,\abs*{\nabla v}_g^2 \eta^2 \, d \Vg + \frac{1}{2} \int_{M} v^{-1} \eta^2 \, d \Vg \leq \frac{1}{2} \int_{M} v^{-1} \,\abs*{\nabla v}_g^2 \eta^2 \, d \Vg + 2 \int_{M} v \,\abs*{\nabla \eta}_g^2 \, d \Vg,
		\end{equation*}
		which, using the properties of~$\eta$ in~\eqref{eq:cutoff}, implies
		\begin{equation*}
			\int_{B_R} P \, d\Vg \leq \int_{M} P \eta^2 \, d \Vg \leq 4 \int_{M} v \,\abs*{\nabla \eta}_g^2 \, d \Vg \leq \frac{4}{R^2} \int_{B_{2R} \setminus B_{R}} v \, d \Vg.
		\end{equation*}
		Combining the latter with~\eqref{eq:buv} entails that
		\begin{equation*}
			\int_{B_R} P \, d\Vg \leq 16 \, \Vg\!\left(B_{2R}\right) e^{q^\beta(2R) F^\beta(2R)} \quad\text{for } R \geq c,
		\end{equation*}
		hence
		\begin{equation}
		\label{eq:estP}
			\dashint_{B_R} P \, d\Vg \coloneqq \frac{1}{\Vg(B_{R})} \int_{B_R} P \, d\Vg \leq 16 \,\frac{\Vg(B_{2R})}{\Vg(B_{R})} \, e^{q(2R)^\beta F(2R)^\beta} \leq 4^3 e^{q(2R)^\beta F(2R)^\beta} \!,
		\end{equation}
		where in the latter we used the first inequality in~\eqref{eq:vg}.

		We now shall apply Lemma~\ref{lem:karp} to~$\log P$. To this end, we start by observing that, for every~$\theta>1$,~$c_\theta \geq e^{\theta-1}>1$, and~$t>0$, the function~$\varphi_\theta(t) \coloneqq \log^\theta(t+c_\theta)$ is increasing and concave, therefore, Jensen's inequality gives that
		\begin{equation*}
			\dashint_{B_R} \varphi_\theta(P) \, d\Vg \leq \varphi_\theta \!\left(\dashint_{B_R} P \, d\Vg\right) \quad\text{for every } R>0,
		\end{equation*}
		which reads as
		\begin{equation}
		\label{eq:estL}
			\dashint_{B_R} \log^\theta(P+c_\theta) \, d\Vg \leq \log^\theta \!\left(\dashint_{B_R} P \, d\Vg + c_\theta\right) \quad\text{for every } R>0.
		\end{equation}
		Hence, combining~\eqref{eq:estP}--\eqref{eq:estL}, it follows that
		\begin{equation*}
			\dashint_{B_R} \left(\log P\right)_{+}^\theta \, d\Vg \leq \dashint_{B_R} \log^\theta(P+c_\theta) \, d\Vg \leq \log^\theta \!\left(4^3 e^{q^\beta(2R) F^\beta(2R)} + c_\theta\right) \quad\text{for } R \geq c
		\end{equation*}
		and since~$q$ and~$F$ are positive, we obtain
		\begin{equation*}
			\begin{split}
				\dashint_{B_R} \left(\log P\right)_{+}^\theta \, d\Vg &\leq \log^\theta \!\left(2 \max\!\left\{64,c_\theta\right\} e^{q^\beta(2R) F^\beta(2R)}\right) \\
				&\leq 2^{\theta-1} \!\left[\log^\theta \!\left(2 \max\!\left\{64,c_\theta\right\}\right) + q^{\beta\theta}(2R) F^{\beta\theta}(2R)\right] \quad\text{for } R \geq c.
			\end{split}
		\end{equation*}
		Finally, recalling that~$\beta \in (0,1)$ and choosing~$\theta=\frac{1}{\beta}>1$, we infer that
		\begin{equation}
		\label{eq:toest}
			\dashint_{B_R} \left(\log P\right)_{+}^\theta \, d\Vg \leq 2^{\theta-1} \!\left[\mathsf{c}_\theta + q(2R) F(2R)\right] \quad\text{for } R \geq c,
		\end{equation}
		where~$\mathsf{c}_\theta \coloneqq \log^\theta \!\left(2 \max\!\left\{64,c_\theta\right\}\right)>1$.
		
		Recalling the definition of~$q$ in~\eqref{eq:defq} and~\eqref{eq:vg}, from~\eqref{eq:toest} we deduce
		\begin{align*}
			\int_{B_R} \left(\log P\right)_{+}^\theta \, d\Vg &\leq 2^{\theta-1} \!\left[\mathsf{c}_\theta + \frac{4R^2}{V(2R)} F(2R)\right] V(R) \leq 2^{\theta-1} \!\left[\mathsf{c}_\theta \, V(R) + 4R^2 F(2R)\right] \\
			&\leq 2^{\theta-1} \!\left[\pi \mathsf{c}_\theta + 4 F(2R)\right] R^2 \leq 2^{\theta+1} \mathsf{c}_\theta \left[1 + F(2R)\right] R^2 \quad\text{for } R \geq c.
		\end{align*}
		
		Since~$F$ is non-decreasing, for~$R \geq c$ we have~$F(2R) \geq F(c)$, which, combined with the latter estimate, entails
		\begin{equation}
		\label{eq:fin}
			\int_{B_R} \left(\log P\right)_{+}^\theta \, d\Vg \leq \mathsf{c}_{c,\theta} R^2 F(2R) \quad\text{for } R \geq c,
		\end{equation}
		for some~$\mathsf{c}_{c,\theta}>0$ depending only on~$\theta$ and~$F(c)$.
		
		We now set~$U \coloneqq \log P$ and~$G(R) \coloneqq \mathsf{c}_{c,\theta} F(2R)$. Since~$F$ is non-decreasing and satisfies~\eqref{eq:hF} by assumption, it follows that~$G$ is also non-decreasing and satisfies~\eqref{eq:hF}, possibly for a larger value of~$c$. Moreover, by Lemma~\ref{lem:DP-Dln}, we also have~$\Delta_g U \geq 0$ in~$M$. Since~\eqref{eq:fin} can be rewritten as
		\begin{equation*}
			\int_{B_R} U_{+}^\theta \, d\Vg \leq R^2 G(R),
		\end{equation*}
		we are in a position to apply Lemma~\ref{lem:karp} and infer that~$U_{+}$ is constant in~$M$. This is equivalent to~$\left(\log P\right)_{+} = \kappa \geq 0$ in~$M$. If~$\kappa > 0$, we immediately conclude that~$P = e^\kappa$ in~$M$; otherwise we have~$0<P \leq 1$ in~$M$. In this case, the volume growth~\eqref{eq:vg}, together with Corollary~1 in~\cite{cy}, implies that~$M$ is parabolic. Since~$P$ is bounded and subharmonic by Lemma~\ref{lem:DP-Dln}, it must be constant in this case as well.
	\end{proof}
	
	We now present the proofs of the last two corollaries.
	
	\begin{proof}[Proof of Corollary~\ref{cor:c1}]
		By the Bishop-Gromov volume comparison theorem, the function~$q$ defined in~\eqref{eq:defq} satisfies
		\begin{equation*}
			\lim_{R \to +\infty} q(R) \in \left[\pi^{-1},+\infty\right] \!,
		\end{equation*}
		hence there exists~$R_0>0$ such that~$q(R) \geq \frac{1}{2\pi}$ for all~$R \geq R_0$. As a consequence, from~\eqref{eq:bbu-1} we obtain
		\begin{equation*}
			\begin{split}
				u(x) &\geq - 4 \log r(x) - \frac{1}{(2\pi)^\beta} \left[2\pi F(r(x)) \right]^\beta \\
				&\geq - 4 \log r(x) - q^\beta(r(x)) G^\beta(r(x)) \quad\text{for } r(x) \geq \max\!\left\{c,R_0\right\} \!,
			\end{split}
		\end{equation*}
		where~$G(R) \coloneqq 2\pi F(R)$. Thanks to the assumptions on~$F$, the function~$G$ satisfies the hypotheses of Theorem~\ref{th:main}, and thus we conclude the desired result.
	\end{proof}
	
	\begin{proof}[Proof of Corollary~\ref{cor:c2}]
		By contradiction, suppose that~\eqref{eq:limsup} does not hold true for some~$\varepsilon \in [0,1-\beta)$, then
		\begin{equation}
		\label{eq:volg}
			V(R) \leq c_\varepsilon  R^{1+\varepsilon} F^{\varepsilon}(R) \quad\text{for every } R \geq R_\varepsilon,
		\end{equation}
		for some~$R_\varepsilon>0$ and~$c_\varepsilon>0$. Recalling~\eqref{eq:defq}, it follows that
		\begin{equation*}
			c_\varepsilon^{\frac{1}{1-\varepsilon}} F^{\frac{\varepsilon}{1-\varepsilon}}(R) \, q^{\frac{1}{1-\varepsilon}}(R) \geq R \quad\text{for every } R \geq R_\varepsilon.
		\end{equation*}
		Hence, from~\eqref{eq:bbu-2}, we obtain
		\begin{equation*}
			\begin{split}
				u(x) &\geq - c_\varepsilon^{\frac{\beta}{1-\varepsilon}} F^{\frac{\varepsilon\beta }{1-\varepsilon}}(r(x)) \, q^{\frac{\beta}{1-\varepsilon}}(r(x)) F^\beta(r(x)) \\ &\geq - 4 \log r(x) - q^{\beta'}\!(r(x)) G^{\beta'}\!(r(x)) \quad\text{for } r(x) \geq \max\!\left\{c,R_\varepsilon\right\} \!,
			\end{split}
		\end{equation*}
		where~$G(R) \coloneqq c_\varepsilon F(R)$ and~$\beta' \coloneqq \frac{\beta}{1-\varepsilon} \in (0,1)$. As above, by Theorem~\ref{th:main}, we conclude that~$(M,g)$ is isometric to~$(\R^2,g_{\R^2})$. Consequently, it follows that
		\begin{equation}
		\label{eq:eu-vg}
			V(R) = \mathsf{c} R^2
		\end{equation}
		for some~$\mathsf{c}>0$. If~$\varepsilon=0$, this immediately contradicts~\eqref{eq:volg}. Hence, we may assume~$\varepsilon>0$. Combining~\eqref{eq:volg}--\eqref{eq:eu-vg}, we obtain
		\begin{equation*}
			F(R) \geq \mathsf{c} R^{\frac{1}{\varepsilon}-1} \quad\text{for every } R \geq R_\varepsilon,
		\end{equation*}
		for a possibly different~$\mathsf{c}>0$. As~$\varepsilon \in (0,1)$, this contradicts~\eqref{eq:hF}.
		
		Finally, since~$M$ is two-dimensional with~$\mathrm{Ric} \geq 0$ it follows that~$K_g \geq 0$. Therefore, the classical results of Cohn–Vossen~\cite{cv} and Huber~\cite{huber} imply that~$M$ is either conformal to the Euclidean plane endowed with the standard metric, or is isometric to the flat cylinder~$\R \times \sfera^1$ or, in the non-orientable case, to the flat open M\"obius strip. Since the latter two alternatives have linear volume growth, which is ruled out by~\eqref{eq:limsup} with~$\varepsilon=0$, the conclusion follows.
	\end{proof}
	
	% -------------------------------------------------------------------
	% MODEL MANIF
	
	\section{Intermezzo on concave functions}
	\label{sec:mmanif}
	
	Before proceeding with our analysis, we recall a lemma concerning concave functions.
	
	\begin{lemma}
		\label{lem:lemmino}
		Let~$\tau>0$ and let~$f \in C^1([\tau,+\infty))$ be a concave function such that~$f > 0$ in~$[\tau,+\infty)$. Then, it follows that
		\begin{equation*}
			f' \geq 0 \quad\text{in } [\tau,+\infty)
		\end{equation*}
		and precisely one of the following two alternatives holds:
		\begin{enumerate}[leftmargin=*, label=$(\roman*)$]
			\item\label{it:lem1} we have~$f' > 0$ in~$[\tau,+\infty)$;
			\item\label{it:lem2} there exists~$t_0 \geq \tau$ such that~$f'(t) > 0$ for~$t \in (\tau,t_0)$ and~$f(t) = f(t_0)$ for~$t \in [t_0,+\infty)$.
		\end{enumerate}
	\end{lemma}
	\begin{proof}
		By contradiction, suppose that~$f'(t_\ast)<0$ for some~$t_\ast \geq \tau$. Since~$f$ is positive and concave, it follows that
		\begin{equation}
		\label{eq:conc}
			0 < f(t) \leq f(t_\ast) + f'(t_\ast) \left(t-t_\ast\right) \quad\text{for every } t \geq t_\ast.
		\end{equation}
		Since the right-hand side of~\eqref{eq:conc} vanishes for some~$t \in [t_\ast,+\infty)$, this is impossible. Therefore,~$f' \geq 0$ in~$[\tau,+\infty)$.
		
		Now suppose that there exists~$t_1 \geq \tau$ such that~$f'(t_1) = 0$. By concavity, we have
		\begin{equation*}
			0 \leq f'(t) \leq f'(t_1) = 0 \quad\text{for all } t \geq t_1,
		\end{equation*}
		thus, we set
		\begin{equation}
		\label{eq:def-t0}
			t_0 \coloneqq \inf \!\left\{t \geq \tau \mid f(s)=f(t_1) \;\;\text{for every } s \geq t\right\} \in [\tau,t_1].
		\end{equation}
		By definition,~$f(t) = f(t_0)$ for~$t \in [t_0,+\infty)$. Moreover,~$f'(t) > 0$ for every~$t \in (\tau,t_0)$. Otherwise, as above, we would have~$f(s)=f(t)$ for all~$s \geq t$. Since~$t<t_0 \leq t_1$, this yields~$f(s)=f(t_1)$ for all~$s \geq t$, contradicting the definition of~$t_0$ in~\eqref{eq:def-t0}.
	\end{proof}
	
	We now construct two-dimensional smooth model manifolds with~$\mathrm{Ric} \geq 0$ that is not identically zero and investigate the behavior of~$q$ in this setting in the next two results.
	
	\begin{proposition}
		\label{prop:varmod}
		Let~$\tau>0$ and~$k \in \{2,3,\dots\} \cup \{\infty\}$. Let~$f \in C^k([\tau,+\infty))$ be such that
		\begin{equation}
		\label{eq:assf}
			f>0, \quad f'' \leq 0 \quad\text{in } [\tau,+\infty) \quad\text{and}\quad f'(\tau) <1.
		\end{equation}
		Then, it follows that:
		\begin{enumerate}[leftmargin=*]
			\item\label{it:pm1} there exist~$\psi \colon [0,+\infty) \to [0,+\infty)$,~$\psi \in C^k([0,+\infty))$,~$t_0 > \max\!\left\{\tau, \frac{2}{1-f'(\tau)}\right\}$, and $f_0 \in (0,t_0+1)$ such that
			\begin{gather}
			\label{eq:expr-psi}
				\psi(t)=t \quad\text{for } t \in [0,t_0] \quad\text{and}\quad \psi(t)=f(t)+f_0 \quad\text{for } t \in [t_0+1,+\infty), \\
			\label{eq:der-psi}
				0 \leq \psi'(t)\leq 1 \quad\text{and}\quad \psi''(t) \leq 0 \quad\text{for } t \in (0,+\infty) \quad\text{with } \psi'' \not\equiv 0.
			\end{gather}
			
			If, in addition,~$f'>0$ in~$[\tau,+\infty)$, then also~$\psi'>0$ in~$[0,+\infty)$;
			
			\item\label{it:pm2} if~$\psi \in C^\infty([0,+\infty))$ is constructed as in point~(\ref{it:pm1}), the two-dimensional smooth model manifold~$(M,g)$, where~$g = dr^2+ \psi^2(r)\, d\theta^2$, satisfies~$\mathrm{Ric} \geq 0$,~$\mathrm{Ric} \not \equiv 0$, and
			\begin{align}
			\label{eq:vol-R}
				V(R) &= 2\pi \int_{0}^{R} \psi(s) \, ds \quad\text{for every } R>0, \\
			\notag
				V(R) &= 2\pi \left[f_1+f_0 R + \int_{t_0+1}^{R} f(s) \, ds\right] \quad\text{for every } R \geq t_0+1,
			\end{align}
			with~$f_1 \in \R$ satisfying~$\abs*{f_1} \leq \frac{t_0^2}{2}+t_0+1$.
		\end{enumerate}
	\end{proposition}
	\begin{proof}
		We address point~\eqref{it:pm1}. We start by observing that Lemma~\ref{lem:lemmino} and~\eqref{eq:assf} yield
		\begin{equation}
		\label{eq:sign-f'}
			f'(t) \geq 0 \quad\text{for all } t \geq \tau.
		\end{equation}
		By concavity of~$f$, we have~$f(t) \leq f(\tau) + f'(\tau) \left(t-\tau\right)$ for every~$t \geq \tau$. Taking into account~\eqref{eq:sign-f'}, we have~$\tau f'(\tau) \geq 0$, which implies
		\begin{equation*}
			\frac{f(t)}{t} \leq \frac{f(\tau)}{t} + f'(\tau) \quad\text{for every } t \geq \tau.
		\end{equation*}
		Since~$f'(\tau) \in [0,1)$ by~\eqref{eq:assf} and~\eqref{eq:sign-f'}, there exists~$t_0 > \max\!\left\{\tau, \frac{2}{1-f'(\tau)}\right\}$ such that
		\begin{equation}
		\label{eq:est-t0}
			f(t) \leq \frac{1+f'(\tau)}{2} \, 
			t \quad\text{for every } t \geq t_0.
		\end{equation}
		
		Now let~$\varphi \in C^\infty([0,+\infty))$ satisfy~$0 \leq \varphi \leq 1$,~$\varphi' \geq 0$ in~$[0,+\infty)$,~$\varphi = 0$ in~$[0,t_0]$, and~$\varphi =1$ in~$[t_0+1,+\infty)$. Set
		\begin{equation*}
			\sigma(t) \coloneqq 1 - \varphi(t) \left(1-f'(t)\right) \quad\text{for } t \in [0,+\infty),
		\end{equation*}
		which is well defined since~$t_0 > \tau$ and~$\sigma \in C^{k-1}([0,+\infty))$. Moreover,~$\sigma = 1$ in~$[0,t_0]$ and~$\sigma = f'\geq 0$ in~$[t_0+1,+\infty)$ by~\eqref{eq:sign-f'}. We also observe that, since~$f$ is concave, we have~$0 \leq f'(t) \leq f'(\tau) < 1$ for all~$t \geq \tau$ by~\eqref{eq:assf} and~\eqref{eq:sign-f'}, therefore
		\begin{equation}
		\label{eq:f'}
			1-f'(t)>0 \quad\text{for all } t \geq \tau.
		\end{equation}
		From this,~\eqref{eq:assf}, the definition of~$\varphi$, and the fact that~$t_0>\tau$, we deduce that
		\begin{equation*}
			\sigma \geq 0 \quad\text{and}\quad \sigma' \leq 0 \quad\text{in } [0,+\infty).
		\end{equation*}
		We then define
		\begin{equation*}
			\psi(t) \coloneqq \int_{0}^{t} \sigma(s) \, ds,
		\end{equation*}
		which satisfies all the required properties. Indeed,~$\psi \in C^{k}([0,+\infty))$ with~$\psi'(t)=\sigma(t) \in [0,1]$ and~$\psi''(t)=\sigma' (t) \leq 0$ for $t >0$. From this, we see that~$\psi'(0)=\sigma(0)=1$, and~$\psi''(0)=\sigma'(0)=0$. Moreover, by definition of~$\sigma$ and~$\varphi$, we have~$\psi(t)=t$ for~$t \in [0,t_0]$ while, for~$t>t_0+1$,
		\begin{equation}
		\label{eq:psi-tlarge}
			\begin{split}
				\psi(t) &= t_0 + \int_{t_0}^{t_0+1} 1 - \varphi(s) \left(1-f'(s)\right) \, ds + \int_{t_0+1}^{t} f'(s) \, ds \\
				&= t_0 +1 - \int_{t_0}^{t_0+1} \varphi(s) \left(1-f'(s)\right) \, ds + f(t) - f(t_0+1) \\
				&= f(t) - I + t_0 +1 - f(t_0+1),
			\end{split}
		\end{equation}
		where
		\begin{equation*}
			I \coloneqq \int_{t_0}^{t_0+1} \varphi(s) \left(1-f'(s)\right) \, ds \in [0,1].
		\end{equation*}
		Indeed, recalling~\eqref{eq:f'}, we have~$0 \leq \varphi \left(1-f'\right) \leq \varphi \leq 1$ in~$(t_0,t_0+1)$. Hence, we set
		\begin{equation*}
			f_0 \coloneqq - I + t_0 +1 - f(t_0+1) < t_0 +1.
		\end{equation*}
		On the other hand, taking advantage of~\eqref{eq:assf} and~\eqref{eq:est-t0}, we obtain
		\begin{equation*}
			f_0 \geq -1 + t_0 +1 - \frac{1+f'(\tau)}{2} \left(t_0 +1 \right) = - 1 + \frac{1-f'(\tau)}{2} \left(t_0 +1 \right) > - 1 + \frac{1-f'(\tau)}{2} \, t_0 > 0,
		\end{equation*}
		where we used the fact that~$t_0 > \frac{2}{1-f'(\tau)}$ for the last inequality.
		
		We now prove that~$\psi'' \not\equiv 0$. Suppose, by contradiction that~$\psi'' = \sigma' \equiv 0$, then~$\sigma = \sigma(0)=1$. As a result~$\psi(t)=t$ and, from~\eqref{eq:psi-tlarge}, also~$f(t)=t-f_0$ for all~$t \geq t_0+1$. This implies~$f'(t)=1$ for all~$t \geq t_0+1$ and, since~$f$ is concave, that~$f'(\tau) \geq f'(t_0+1) = 1$, which contradicts~\eqref{eq:assf}.
		
		Finally, assume in addition that~$f'>0$. Then,~$\sigma = f'> 0$ in~$[t_0+1,+\infty)$. Since~$\sigma$ is non-increasing, this implies that~$\sigma > 0$ in~$[0,+\infty)$ which, in turn, yields~$\psi' > 0$ in~$[0,+\infty)$. This completes the proof of point~\eqref{it:pm1}. \newline
		
		We now address point~\eqref{it:pm2}. We notice that the Gaussian curvature is given by
		\begin{equation}
		\label{eq:gaussian}
			K_g = - \frac{\psi''(r)}{\psi(r)}.
		\end{equation}
		Since~$M$ is two-dimensional, this and~\eqref{eq:der-psi} yield~$\mathrm{Ric} \geq 0$ with~$\mathrm{Ric} \not \equiv 0$.
		
		Since~$\psi$ is non-negative, the volume form is~$d\Vg = \psi(r) \, dr \wedge d\theta$, which gives~\eqref{eq:vol-R}. Moreover, for~$R \geq t_0+1$, using~\eqref{eq:expr-psi} we have
		\begin{align*}
			V(R) &= 2\pi \left[\int_{0}^{t_0} s \, ds + \int_{t_0}^{t_0+1} \psi(s) \, ds + \int_{t_0+1}^{R} f(s) + f_0 \, ds\right] \\
			&= 2\pi \left[\frac{t_0^2}{2} + \int_{t_0}^{t_0+1} \psi(s) \, ds +f_0 \left(R- (t_0+1)\right) + \int_{t_0+1}^{R} f(s) \, ds\right] \\
			&= 2\pi \left[f_1 + f_0 R + \int_{t_0+1}^{R} f(s) \, ds\right] \!,
		\end{align*}
		where~$f_1 \coloneqq \frac{t_0^2}{2} + \int_{t_0}^{t_0+1} \psi(s) \, ds - f_0 \left(t_0+1\right)$. By~\eqref{eq:expr-psi}--\eqref{eq:der-psi} and since~$f_0 \in (0,t_0+1)$,
		\begin{equation*}
			t_0 = \psi(t_0) \leq \int_{t_0}^{t_0+1} \psi(s) \, ds \leq \psi(t_0+1) \leq t_0+1,
		\end{equation*}
		hence, we have~$f_1 \leq \frac{t_0^2}{2} + t_0+1$ and also
		\begin{equation*}
			f_1 \geq \frac{t_0^2}{2} + t_0 - \left(t_0+1\right)^2 = - \left(\frac{t_0^2}{2} + t_0+1\right) \!.
		\end{equation*}
	\end{proof}
	
		\begin{lemma}
		\label{lem:q-psi}
		The following assertions hold true.
		\begin{enumerate}[leftmargin=*]
			\item\label{it:l1} Let~$\psi \in C^0((0,+\infty))$ be such that~$\psi>0$ in~$(0,+\infty)$ and let~$t_0 > 0$.
			Then, we have
			\begin{equation*}
				\frac{t^2}{\int_{t_0}^{t} \psi(s) \, ds} \leq \left(1+\frac{t_0}{t-t_0}\right)^{\!\! 2} \int_{t_0}^{t} \frac{ds}{\psi(s)} \quad\text{for every } t>t_0.
			\end{equation*}
			
			\item\label{it:l2} Let~$\psi \in C^0([0,+\infty)) \cap C^1((0,+\infty))$ be a concave function such that 
			$\psi>0$ in $(0,+\infty)$ and let~$t_0 > 0$. Then, it follows that
			\begin{equation*}
				\int_{t_0}^{t} \frac{ds}{\psi(s)} \leq \frac{t^2}{\int_{0}^{t} \psi(s) \, ds} \left(\log t - \log t_0\right)  \quad\text{for every } t>t_0.
			\end{equation*}
			
			\item\label{it:l3} Let $(M,g)$ be a two-dimensional smooth model manifold~\footnote{See the beginning of Section~\ref{sec:radsol} for the definition and properties of such a Riemannian manifold.}, with~$g = dr^2+ \psi^2(r)\, d\theta^2$, and let~$t_0 > 0$. Then, it follows that
			\begin{equation*}
				q(t) \leq \frac{1}{2 \pi} \left(1+\frac{t_0}{t-t_0}\right)^{\!\! 2} \int_{t_0}^{t} \frac{ds}{\psi(s)} \quad\text{for every } t>t_0
			\end{equation*}
			and
			\begin{equation*}
				\int_{t_0}^{t} \frac{ds}{\psi(s)} \leq 2 \pi q(t) \left(\log t - \log t_0\right)  \quad\text{for every } t>t_0,
			\end{equation*}
			where~$q$ is defined in~\eqref{eq:defq}.
		\end{enumerate}
	\end{lemma}
	\begin{proof}
		For point~\eqref{it:l1}, let~$t>t_0>0$. An application of H\"older inequality yields
		\begin{equation*}
			(t-t_0)^2 = \left(\int_{t_0}^{t} \sqrt{\psi(s)} \, \frac{1}{\sqrt{\psi(s)}} \, ds\right)^{\!\! 2} \leq \int_{t_0}^{t} \psi(s) \, ds \int_{t_0}^{t} \frac{ds}{\psi(s)},
		\end{equation*}
		from which
		\begin{equation*}
			%\frac{t^2}{\int_{0}^{t} \psi(s) \, ds} \leq
			\frac{t^2}{ \int_{t_0}^{t} \psi(s) \, ds} = \frac{t^2}{\left(t-t_0\right)^2} \frac{ \left(t-t_0\right)^2}{\int_{t_0}^{t} \psi(s) \, ds} \leq \left(1+\frac{t_0}{t-t_0}\right)^{\!\! 2} \int_{t_0}^{t} \frac{ds}{\psi(s)}.
		\end{equation*}
		This establishes point~\eqref{it:l1} of the lemma. \newline
		
		We then consider point~\eqref{it:l2}. We first observe that, since~$\psi$ is concave with~$\psi(0) \geq 0$
		\begin{equation}
			\label{eq:noninc}
			\text{the map } t \mapsto \frac{\psi(t)}{t} \text{ is non-increasing in } (0,+\infty)
		\end{equation}
		and, for every~$t>t_0>0$, Jensen's inequality gives
		\begin{equation*}
			\frac{1}{t-t_0} \int_{t_0}^{t} \psi(s) \, ds \leq \psi \!\left(\frac{1}{t-t_0} \int_{t_0}^{t} s \, ds\right) = \psi \!\left(\frac{t+t_0}{2}\right) \leq \psi(t),
		\end{equation*}
		where we used the fact that~$\psi$ is non-decreasing by Lemma~\ref{lem:lemmino} applied with any~$\tau >0$. Hence, the latter estimate implies
		\begin{equation}
			\label{eq:jen}
			\int_{t_0}^{t} \psi(s) \, ds \leq \left(t-t_0\right) \psi(t).
		\end{equation}
		Let~$t>t_0>0$. Using~\eqref{eq:noninc}, we obtain
		\begin{equation*}
			\int_{t_0}^{t} \frac{ds}{\psi(s)} = \int_{t_0}^{t} \frac{s}{\psi(s)} \frac{1}{s} \, ds \leq \frac{t}{\psi(t)} \int_{t_0}^{t} \frac{ds}{s} = \frac{t}{\psi(t)} \left(\log t - \log t_0\right) \!,
		\end{equation*}
		which, coupled with~\eqref{eq:jen}, entails
		\begin{equation}
			\label{eq:1/psi}
				\int_{t_0}^{t} \frac{ds}{\psi(s)} \leq \frac{t \left(\log t - \log t_0\right) \left(t-t_0\right)}{\int_{t_0}^{t} \psi(s) \, ds} = \frac{t^2}{\int_{0}^{t} \psi(s) \, ds} \frac{\int_{0}^{t} \psi(s) \, ds}{\int_{t_0}^{t} \psi(s) \, ds} \left(\frac{t-t_0}{t}\right) \left(\log t - \log t_0\right) \!.
		\end{equation}
		We now estimate the integrals on the right-hand side of~\eqref{eq:1/psi}. To this end, since~$\psi$ is non-decreasing, we observe that~$\psi(s) \leq \psi(t_0)$ for~$0<s<t_0$ and~$\psi(s) \geq \psi(t_0)$ for~$s>t_0$, hence~$\int_{0}^{t_0} \psi(s) \, ds \leq t_0 \psi(t_0)$ and~$\int_{t_0}^{t} \psi(s) \, ds \geq \left(t - t_0\right) \psi(t_0)$. As a result, we have
		\begin{equation*}
			\frac{\int_{0}^{t} \psi(s) \, ds}{\int_{t_0}^{t} \psi(s) \, ds} = 1 + \frac{\int_{0}^{t_0} \psi(s) \, ds}{\int_{t_0}^{t} \psi(s) \, ds} \leq 1 + \frac{t_0}{t-t_0} = \frac{t}{t-t_0}.
		\end{equation*}
		Combining the latter with~\eqref{eq:1/psi} yields the desired conclusion. \newline
		
		Finally, point~\eqref{it:l3} follows from Proposition~\ref{prop:varmod}, the definition of~$q$ in~\eqref{eq:defq}, and points~\eqref{it:l1}--\eqref{it:l2} of the lemma.
	\end{proof}
	
	\begin{remark}
		We observe that the logarithmic dependence in point~\eqref{it:l2} of Lemma~\ref{lem:q-psi} is optimal, as one can verify by taking~$\psi(t)=at+f_0$, with~$a \in (0,1)$ and~$f_0>0$. Moreover, a two-dimensional smooth model manifold corresponding to such a~$\psi$ exists by Proposition~\ref{prop:varmod}, as follows by choosing~$f(t)=at$.
	\end{remark}
	
	% -------------------------------------------------------------------
	% RADIAL SOL
	
	\section{Radial solutions on two dimensional smooth model manifolds}
	\label{sec:radsol}
	Let~$(M,g)$ be a two-dimensional smooth model manifold, namely a Riemannian manifold possessing a \textit{pole}~$\mathcal{o} \in M$ and whose metric, in polar coordinates~$(r,\theta)$ around~$\mathcal{o}$, has the expression
	\begin{equation*}
		g=dr^2+\psi^2(r) \, d\theta^2 \quad \text{on } M\setminus\left\{\mathcal{o}\right\} \!,
	\end{equation*}
	where~$d\theta^2$ denotes the usual metric on~$\sfera^{1}$ and the function~$\psi\in C^\infty\!\left(\left[0,+\infty\right)\right)$ satisfies, for any non-negative integer~$j \in \N$,
	\begin{equation}
	\label{psi-nec2}
		\psi^{(2j)}(0)=0,\quad\psi'(0)=1,\quad\text{and}\quad\psi'(r) \geq 0\quad\text{for all } r>0.
	\end{equation}
	
	Then,~$u=u(r)$ is a smooth radial solution of~\eqref{eq:Liou} on the model manifold~$M$ if and only if it satisfies the initial value problem
	\begin{equation}
		\label{eq:ivp-u}
		\begin{cases}
			\begin{aligned}
				&-u''(r) - \frac{\psi'(r)}{\psi(r)} \, u'(r)= e^{u(r)}	&& \text{for } r>0, \\
				&u(0) = \alpha \in \R, \\
				&u'(0) = 0.
			\end{aligned}
		\end{cases}
	\end{equation}
	
	If~$\psi'>0$ in~$(0,+\infty)$, then Proposition~2.1 in~\cite{berchio} guarantees the existence of a smooth radial solution to~\eqref{eq:Liou} on~$(M,g)$ for every~$\alpha \in \R$. Moreover, the maps~$r \mapsto u'(r)$ and~$r \mapsto e^{u(r)}$ are bounded on~$[0,+\infty)$ with~$u'(r)<0$ for~$r>0$.
	
	We observe that Proposition~2.1 in~\cite{berchio} remains valid under the weaker assumption~$\psi' \geq 0$ in~$(0,+\infty)$. Indeed, the existence of a local solution to~\eqref{eq:ivp-u} in~$[0,\bar{r}]$, for some~$\bar{r}>0$, follows from the condition~$\psi'(0)=1$, which is required for the metric~$g$ to be well defined at the origin. The remainder of the proof in~\cite{berchio} then carries over under the assumption~$\psi' \geq 0$ in~$(0,+\infty)$.
	
	In the subsequent result, we consider the case in which~$\psi$ has at most linear growth at infinity with~$\psi' \geq 0$, and study the properties of radial solutions to~\eqref{eq:Liou}.
	
	\begin{proposition}
	\label{prop:ling}
		Let~$(M,g)$ be a two-dimensional smooth model manifold, where~$g = dr^2+ \psi^2(r)\, d\theta^2$ with~$\psi$ satisfying~\eqref{psi-nec2}. Assume, in addition, that there exist~$c_0>0$ and~$r_0>0$ such that
		\begin{equation}
		\label{eq:lg}
			\psi(r) \leq c_0 r \quad\text{for every } r \geq r_0.
		\end{equation}
		Then, it follows that:
		\begin{enumerate}[leftmargin=*]
			\item\label{it:lg1} setting~$I_\infty \coloneqq \int_{0}^{+\infty} \psi(s) e^{u(s)} \, ds$, we have~$I_\infty < +\infty$ and
			\begin{gather}
			\label{eq:lim1}
				\psi(r) u'(r) \to - I_\infty \in (-\infty,0) \quad\text{as } r \to +\infty, \\
			\label{eq:L1}
				e^u \in L^1(M) \quad\text{with } \int_{M} e^u \, d\Vg = 2\pi I_\infty;
			\end{gather}
			
			\item\label{it:lg2} for every~$\bar{r}>0$, we have
			\begin{equation*}
				u(r) \sim - I_\infty \int_{\bar{r}}^{r} \frac{ds}{\psi(s)} \quad\text{as } r \to +\infty.
			\end{equation*}
		\end{enumerate}
	\end{proposition}
	\begin{proof}
		For point~\eqref{it:lg1}, we argue by contradiction. Suppose that~$I_\infty = +\infty$. Then, there exists~$r_1>r_0+1$ such that
		\begin{equation}
		\label{eq:ass-hyp}
			\int_{0}^{r} \psi(s) e^{u(s)} \, ds \geq 10 c_0 \quad\text{for every } r \geq r_1.
		\end{equation}
		We first observe that the equation in~\eqref{eq:ivp-u} can be rewritten as
		\begin{equation*}
			\left(\psi(r) u'(r)\right)' = - \psi(r) e^{u(r)} \quad\text{for } r>0,
		\end{equation*}
		therefore, integrating over~$[0,r]$ and using the fact that~$\psi(0)=0$, we obtain
		\begin{equation}
		\label{eq:eq-int}
			\psi(r) u'(r) = - \int_{0}^{r} \psi(s) e^{u(s)} \, ds \quad\text{for every } r>0.
		\end{equation}
		Combining~\eqref{eq:ass-hyp}--\eqref{eq:eq-int} we deduce
		\begin{equation*}
			\psi(r) u'(r) \leq - 10 c_0 \quad\text{for every } r \geq r_1,
		\end{equation*}
		and~\eqref{eq:lg} finally entails
		\begin{equation*}
			u'(r) \leq - \frac{10}{r} \quad\text{for every } r \geq r_1.
		\end{equation*}
		Integrating this inequality and exponentiating the resulting estimate, we obtain
		\begin{equation*}
			e^{u(r)} \leq \mathsf{c} r^{-10} \quad\text{for every } r \geq r_1,
		\end{equation*}
		where~$\mathsf{c}>0$ depends on~$r_1>0$. As a consequence, combining this estimate with~\eqref{eq:lg}, we get
		\begin{equation*}
			\psi(r) e^{u(r)} \leq \mathsf{c} c_0 r^{-9} \quad\text{for every } r \geq r_1.
		\end{equation*}
		This immediately leads to a contradiction of the fact that~$I_\infty = +\infty$.
		
		Hence, we conclude that~$I_\infty < +\infty$. Moreover,~\eqref{eq:eq-int} immediately yields~\eqref{eq:lim1}. To prove~\eqref{eq:L1}, we simply observe that
		\begin{equation*}
			\int_{M} e^u \, d\Vg = 2 \pi \int_{0}^{+\infty} \psi(s) e^{u(s)} \, ds = 2\pi I_\infty.
		\end{equation*}
		This completes the proof of point~\eqref{it:lg1}. \newline
		
		Finally, point~\eqref{it:lg2} follows from~\eqref{eq:lim1}, together with the fact that, by~\eqref{eq:lg}, we have
		\begin{equation*}
			\int_{\bar{r}}^{+\infty} \frac{ds}{\psi(s)} = +\infty.
		\end{equation*}
	\end{proof}
	
	\begin{remark}
		For the smooth model manifolds constructed in Proposition~\ref{prop:varmod}, thanks to~\eqref{eq:expr-psi}--\eqref{eq:der-psi}, we have~$\psi(t) \leq t$ and~$\psi' \geq 0$. Thus, Proposition~\ref{prop:ling} applies in this case.
	\end{remark}
	
	We conclude this section with the proof of Corollary~\ref{cor:c3}.
	
	\begin{proof}[Proof of Corollary~\ref{cor:c3}]
		Since~$\mathrm{Ric} \geq 0$,~$\mathrm{Ric} \not \equiv 0$, and~$M$ is two-dimensional, we deduce from~\eqref{eq:gaussian} that~$\psi'' \leq 0$ in~$(0,+\infty)$ with~$\psi'' \not\equiv 0$. Therefore, from the properties of~$\psi$ in~\eqref{psi-nec2} and Lemma~\ref{lem:lemmino}, we have
		\begin{equation*}
			\psi'(t) \geq 0 \quad\text{and}\quad \psi(t) \leq t \quad\text{for } t \in (0,+\infty).
		\end{equation*}
		Hence, the assumptions of Proposition~\ref{prop:ling} are satisfied. Thus, from point~\eqref{it:lg2} of Proposition~\ref{prop:ling} with~$\bar{r}=1$, we have
		\begin{equation*}
			u(r) \sim - I_\infty \int_{1}^{r} \frac{ds}{\psi(s)} \quad\text{as } r \to +\infty.
		\end{equation*}
		From this, exploiting point~\eqref{it:l2} of Lemma~\ref{lem:q-psi} with~$t_0=1$, it follows that
		\begin{equation*}
			u(r) \geq - \gamma q(r) \log r \quad\text{for every } r \geq r_0,
		\end{equation*}
		for some~$\gamma >0$ and~$r_0 \geq 1$. As a consequence, we have
		\begin{equation*}
			u(r) \geq - 4 \log r - \gamma q(r) \log r \quad\text{for every } r \geq r_0,
		\end{equation*}
		which implies that~$\beta=1$ in Theorem~\ref{th:main} is not allowed.
	\end{proof}
	
	% -------------------------------------------------------------------
	% EXAMPLES
	
	\section{Optimality of Corollary~\ref{cor:c2}}
	\label{sec:ex}
	
	In this brief section, we illustrate the optimality of Corollary~\ref{cor:c2} by presenting some explicit examples. We begin with the case $\varepsilon=0$ in Corollary~\ref{cor:c2}.
	
	\begin{example}
		Consider the function
		\begin{equation*}
			u(t) = -2 \log\!\left( \cosh\left(\frac{\sqrt{2}}{2}\vert t \vert\right)\right) \!.
		\end{equation*}
		Then,~$u$ is a solution to
		\begin{equation*}
			- u'' = e^u \quad\text{in } \R
		\end{equation*}
		which satisfies~$u(\vert t \vert) \sim -\sqrt{2} \,\abs*{t}$, as~$\abs*{t} \to +\infty$.
		
		Therefore,~$u \colon \R \times \sfera^1 \to \R$ provides a solution to~\eqref{eq:Liou} on the flat cylinder~$\R \times \sfera^1$ such that~$u(x) \sim -\sqrt{2}  r(x)$, as~$r(x) \to +\infty$, which has also finite mass on the flat cylinder. This shows that~$\beta=1$ is not allowed in~\eqref{eq:bbu-2} of Corollary~\ref{cor:c2}.
	\end{example}
	
	We now address the general case.
	
	\begin{example}
		For~$f(r)=r^\varepsilon$ and~$\varepsilon \in [0,1)$, let us consider the smooth function~$\psi$ and the model manifold~$(M,g)$ provided by Proposition~\ref{prop:varmod}. Then point~\eqref{it:lg2} of Proposition~\ref{prop:ling} yields
		\begin{equation}
			\label{eq:asym-u}
			u(r) \sim - I_\infty \int_{t_0+1}^{r} \frac{ds}{\psi(s)} \sim - \gamma_1 r^{1-\varepsilon} \quad\text{as } r \to +\infty,
		\end{equation}
		for some~$\gamma_1>0$. From point~\eqref{it:pm2} of Proposition~\ref{prop:varmod}, we deduce that
		\begin{equation}
			\label{eq:vol}
			V(r) \sim \gamma_2 r^{1+\varepsilon} \quad\text{as } r \to +\infty,
		\end{equation}
		whence~$\gamma_2 q(r) \sim r^{1-\varepsilon}$ as~$r \to +\infty$, with~$\gamma_2>0$. Combining this estimate with~\eqref{eq:asym-u}, we infer that
		\begin{equation*}
			u(r) \sim - \gamma_1 \gamma_2 \, q(r) \quad\text{as } r \to +\infty.
		\end{equation*}
		Thus, we set~$F(r) \coloneqq \gamma_1 \gamma_2$ and deduce, using also~\eqref{eq:vol}, that
		\begin{equation*}
			\frac{V(r)}{r^{1+\varepsilon} F^\varepsilon(r)} \sim \frac{\gamma_2 r^{1+\varepsilon}}{r^{1+\varepsilon} \left(\gamma_1 \gamma_2\right)^\varepsilon} = \frac{\gamma_2^{1-\varepsilon}}{ \gamma_1^\varepsilon} \in (0,+\infty) \quad\text{as } r \to +\infty.
		\end{equation*}
		This shows that the case~$\beta=1$ in~\eqref{eq:bbu-2} of Corollary~\ref{cor:c2} cannot be admitted.
	\end{example}
	
	% -------------------------------------------------------------------
	% ACKNOWLEDG ‘’ “”
	
	\section*{Acknowledgments}
	
	\noindent G.C.\ and M.G.\ are members of the “Gruppo Nazionale per l'Analisi Matematica, la Probabilità e le loro Applicazioni” (GNAMPA) of the “Istituto Nazionale di Alta Matematica” (INdAM, Italy). G.C.\ and M.G.\ have been partially supported by the “INdAM - GNAMPA Project”, CUP E53C25002010001.
	
	%-------------------------------------------------------------------
	% BIBLIO

\end{document}